\documentclass[11pt]{amsart} 

\usepackage{amsmath}
\usepackage{amssymb}
\usepackage{amsthm}
\usepackage{url}
\usepackage[colorlinks,linkcolor = blue,anchorcolor = red,urlcolor = blue,citecolor = blue]{hyperref}

\theoremstyle{plain}
\newtheorem{theorem}{Theorem}
\newtheorem{lemma}{Lemma}
\newtheorem{prop}{Proposition}
\theoremstyle{definition}
\newtheorem{definition}{Definition}
\newtheorem{example}{Example}
\newtheorem{remark}{Remark}

\usepackage{enumerate}

\usepackage{geometry} 
\author[L.~Guo]{Liang Guo}
\address[L.~Guo]{Centre for Mathematical Sciences, Lund University, Box 118, Lund, 221 00, Sweden}
\email{liang.guo@math.lth.se}

\author[J.~Qian]{Jin Qian$^*$}
\address[J.~Qian]{Research Center for Operator Algebras, School of Mathematical Sciences, East China Normal University, Shanghai 200241, PR China}
\email{52265500013@stu.ecnu.edu.cn}
\thanks{$^*$Corresponding author.}

\author[Q.~Wang]{Qin Wang}
\address[Q.~Wang]{Research Center for Operator Algebras, School of Mathematical Sciences, East China Normal University, Shanghai 200241, PR China.}
\email{qwang@math.ecnu.edu.cn}



\title[A nonstandard analysis approach to limit operators]{A nonstandard analysis approach to limit operators and Fredholmness in Roe-like algebras}

\thanks{This is the accepted version of the following article: Guo, L., Qian, J. and Wang, Q. (2026), A nonstandard analysis approach to limit operators and Fredholmness in Roe-like algebras. Bull. London Math. Soc., 58: e70534, which has been published in final form at https://doi.org/10.1112/blms.70534.}

\begin{document}

\begin{abstract}

Let $(X,d)$ be a uniformly locally finite metric space, and $T$ an operator in the uniform Roe algebra $C_u^*(X)$ (or uniform quasi-local algebra $C_{ql}^*(X)$). In this paper, we introduce the concept of limit operators of $T$ on galaxies in the nonstandard extension of $X$, and prove that $T$ is a generalized Fredholm operator with respect to the ghost ideal in $C_u^*(X)$ (or $C_{ql}^*(X)$) if and only if all limit operators on afar galaxies are invertible, and their inverses are uniformly bounded. In particular, if $X$ has Yu's Property A, then $T$ is a Fredholm operator if and only if all limit operators on afar galaxies are invertible. Using techniques in nonstandard analysis, our result strengthens a work of Špakula--Willett \cite{SpW} on the characterization of Fredholmness by using fewer limit operators.

\end{abstract}

\maketitle

\section{Introduction}

The study of limit operators of band-dominated operators on $\mathbb{Z}^n$ was initiated in \cite{RRS04}. It was broadened to the class of band-dominated operators on discrete groups with Yu's Property A by J.~Roe in \cite{Roe05}. They proved that a band-dominated operator is Fredholm if and only if all its limit operators are invertible and their inverses are uniformly bounded. The condition of uniform boundedness was later proved to be satisfied automatically, hence can be removed in \cite{LS14} in the case of $\mathbb{Z}^n$. Their arguments can be further generalized to discrete groups with Property A (see \cite{Roeblog}).

In \cite{SpW}, J. Špakula and R. Willett further generalized limit operators to band-dominated operators on strongly discrete metric spaces with bounded geometry and Property A. They defined limit spaces for such a metric space, and associated to each limit space a limit operator of a given band-dominated operator. It was proved that a band-dominated operator is Fredholm if and only if all its limit operators are invertible \cite{SpW}. 

In this paper, we exploit the techniques from nonstandard analysis to study limit operators. Let $(X,d)$ be a uniformly locally finite metric space, and $({}^*X,{}^*d)$ its nonstandard extension (for all nonstandard analysis terms in the paragraph, see Section 2.1). For $x,y\in{}^*X$, ${}^*d(x,y)$ being finite defines an equivalence relation on ${}^*X$. An equivalence class under this equivalence relation is called a \emph{galaxy}, and a galaxy infinitely distant away from $X$ is called \emph{afar}. For an operator $T$ on $\ell^2(X)$ with kernel function $k_T$, the standard part of the nonstandard extension of $k_T$, restricting its domain to $G\times G$, gives a kernel function of an operator $T_G$ on every galaxy $G$, i.e., $k_{T_G}:=st\circ {}^{*}k_T|_{G\times G}$. For an afar galaxy $G$, we call $T_G$ the \emph{limit operator} of $T$ on $G$. 

We prove the following main result in this paper. 

\begin{theorem}

Let $(X,d)$ be a uniformly locally finite metric space and $T \in C^*_u(X)$ (respectively, $T\in C^*_{ql}(X)$).

\begin{enumerate}

\item The operator $T$ is a generalized Fredholm operator with respect to the ghost ideal in $C^*_u(X)$ (respectively, $C^*_{ql}(X)$) if and only if all limit operators of $T$ on afar galaxies are invertible, and their inverses are uniformly bounded.
\item If $X$ has Yu's Property A, then $T$ is a Fredholm operator if and only if all limit operators of $T$ on afar galaxies are invertible.

\end{enumerate}

\end{theorem} 

A galaxy in $^{*}X$ is bijectively isometric to a limit space defined in \cite{SpW} via the hyper-principal map (see Section 5). Moreover, this bijection induces a unitary equivalence between the limit operators we defined and those introduced in \cite{SpW}. It is possible to choose a ``large'' ${}^*X$ so that every limit space can be realized as a galaxy, while in a ``small'' $^{*}X$, some limit spaces might be omitted (see Remark 3). In the latter case, the theorem allows us to prove Fredholmness by showing the invertibility of fewer limit operators, which enhances the result in \cite{SpW}. 

We fix a countably incomplete ultrafilter $\omega$ on an index set $I$, and use $X^\omega$, the set-theoretic (discrete) ultrapower of $X$, instead of the abstract nonstandard extension ${}^*X$. We will use diagonal arguments in ultrapowers instead of saturation arguments. Our proofs can be converted back into saturation arguments, thereby establishing the validity of the results for general nonstandard extensions. 

The paper is organized as follows. In Section 2, we briefly recall some basic concepts in nonstandard analysis and coarse geometry. In Section 3, we discuss the universal uniformly locally finite coarse structure on a set $X$ and introduce a new Roe-like algebra on $X$, so-called block diagonal algebras. In Section 4, we prove our main result Theorem 1(1) on generalized Fredholmness with respect to the ghost ideal. In Section 5, we compare our approach to limit operators with that in \cite{SpW}. In Section 6, we use countable saturation of ultraproduct from nonstandard analysis to remove the condition of uniform boundedness, when the underlying space has Yu's Property A.

\section{Preliminaries}

\subsection{Nonstandard analysis}

In this section, we recall some concepts in nonstandard analysis used in this paper (cf.~\cite{Go98}). 

Throughout this paper, $\omega$ stands for a countably incomplete ultrafilter on an index set $I$. Here an ultrafilter $\omega$ is countably incomplete if the index set $I$ can be decomposed into $I=\cup_{n\in\mathbb{N}}I_n$ such that $I_n$'s are pairwise disjoint and all $I_n$'s are not in $\omega$. Equivalently, there exists a decreasing sequence of subsets of $I$, i.e., $I\supset J_0\supset J_1\supset\cdots\supset J_n\supset\cdots$, such that all $J_n$'s are in $\omega$, while $\cap_{n\in\mathbb{N}}J_n=\emptyset$. We focus on the ultrapower model of nonstandard extension, i.e., by a nonstandard extension of $X$, we mean $X^\omega$, the set-theoretic ultrapower of $X$.

Let $\mathbb{R}^\omega$ be the set-theoretic ultrapower of real number set $\mathbb{R}$. Elements in $\mathbb{R}^\omega$ are called hyperreals. The set of real numbers $\mathbb{R}$ embeds in $\mathbb{R}^\omega$ by mapping $r$ to $[r]$,  the equivalence class to which the sequence indexed $i$ with constant value $r$ belongs. All functions and relations on $\mathbb{R}$ extend on $\mathbb{R}^\omega$ by $f^\omega([r_i]):=[f(r_i)]$. For example, the absolute value extends to $\mathbb{R}^\omega$ by $\big|[r_i]\big|^\omega=\big[|r_i|\big]$. For simplicity, we omit the upper index $\omega$ for operators like $+$, $>$, and regard $\mathbb{R}$ as a subset of $\mathbb{R}^\omega$. For example, for a hyperreal $x>^{\omega}[0]$, we write simply $x>0$.

By Łoś' theorem, a first order sentence is true in the ultrapower model if and only if it is true in the original model. Hence $\mathbb{R}^\omega$ is an ordered field. A hyperreal whose absolute value is smaller than any positive real number is called an infinitesimal, while a hyperreal whose absolute value is greater than any positive real number is said to be infinite. The set of hyperreals $\mathbb{R}^\omega$ contains nonzero infinitesimals and infinite numbers. For a finite (i.e. not infinite) hyperreal $x$, there exists a unique real number, called the standard part of $x$ and denoted by $st(x)$, such that $x-st(x)$ is an infinitesimal. The value of $st([x_i])$ is actually $\lim _{i\to\omega}x_i$.

Let $A$ be a countable set of infinite positive hyperreals, then by so-called \emph{countable saturation} of $\mathbb{R}^\omega$, $A$ has an infinite lower bound. This property will be used in Section 6.

Let $(X,d)$ be a metric space. Its set-theoretic ultrapower $(X^\omega,d^\omega)$ is an  $\mathbb{R}^\omega$-valued metric space. Let $\approx_F$ be the equivalence relation of having finite distance, and $\approx$ the equivalence relation of having infinitesimal distance. An equivalence class under $\approx_F$ is called a galaxy and $st\circ d^\omega$ is a (real-valued) pseudo-metric on it. An element (or a galaxy) in $(X^\omega,d^\omega)$ is called afar if it is infinitely distant from $X$. The set of all afar galaxies in $X^\omega$ is denoted by $AG(X^\omega)$. The galaxy in $X^\omega$ containing $X$ is called the standard galaxy. It is $X$ itself if and only if $X$ is a uniformly locally finite metric space (defined in the following subsection).

\subsection{Coarse geometry}

For a detailed introduction of basic concepts in coarse geometry, we refer the readers to \cite{Roe2003}.

Recall that a binary relation on $X$ is a subset of $X\times X$.
\begin{definition}
Let $X$ be a set. A coarse structure on $X$ is a set $\mathcal{E}$ of binary relations on $X$, such that:

\begin{enumerate}

\item Every finite binary relation belongs to $\mathcal{E}$,
\item $\mathcal E$ is closed under composition, inverse relation, sub-relation, and finite union.

\end{enumerate}
A coarse structure $\mathcal E$ is called unital if the equality relation $\{(x,x)\mid x\in X\}$ belongs to $\mathcal{E}$.
\end{definition}

Relations in $\mathcal{E}$ are called \emph{entourages}. For an entourage $E$, denote $E_x:=\{y\in X\mid (y,x)\in E\}$ and ${}_xE:=\{y\in X\mid (x,y)\in E\}$. A coarse structure is called \emph{uniformly locally finite (ULF)} if for any entourage $E$, there exists $n\in\mathbb{N}$ such that $\#E_x\leq n$ and $\#{}_xE\leq n$ for any $x\in X$, where $\#$ stands for the cardinality of a set.

\begin{example}
Let $(X,d)$ be a pseudo-metric space. There is a natural coarse structure on $X$ given by $\{E\subset X\times X\mid \sup\{d(x,y)\mid(x,y)\in E\}<\infty\}$.
\end{example}

\begin{definition}
Let $(X,\mathcal{E})$ be a coarse space, i.e, a set with a coarse structure, and $T\in B(\ell^2(X))$. 

\begin{enumerate}

\item We say $T$ has \emph{controlled propagation} if the support of its kernel function is an entourage. The set of all controlled propagation operators is denoted by $\mathbb{C}_u[X,\mathcal{E}]$ and its closure with respect to the operator norm in $B(\ell^2(X))$ is called the uniform Roe algebra of $(X,\mathcal{E})$, denoted by $C^*_u(X,\mathcal{E})$. Elements in $C^*_u(X,\mathcal{E})$ are also called \emph{band-dominated operators}.
\item We say $T$ is \emph{quasi-local} if for any $\varepsilon >0$, there exists an entourage $E$ such that for any $Y_1,Y_2\subset X$ with $(Y_1\times Y_2)\cap E=\emptyset$, $\|P _{Y_1}TP _{Y_2}\|<\varepsilon$, where $P _Y$ is the orthonormal projection from $\ell^2(X)$ to $\ell^2(Y)$. The set of all quasi-local operators is called the \emph{uniform quasi-local algebra} on $X$, denoted by $C^*_{ql}(X,\mathcal{E})$.

\end{enumerate}

\end{definition}

We say a metric space $(X,d)$ is ULF if the coarse structure induced by the metric is ULF, equivalently, for any $R>0$, $\sup\{\#B(x,R)\mid x\in X\}<\infty$, where $\#B(x,R)$ is the cardinality of the ball centered at $x$ with radius $R$. For discrete metric spaces, the ULF condition is also called \emph{bounded geometry}.

Note that if $(G,st\circ d^\omega)$ is a galaxy of $X^\omega$ and $[x_i]\in G$, we have $B([x_i],R)\subset [B(x_i,R^\prime)]$ for any real number $R^\prime > R$, where the right hand side is the ultraproduct of $B(x_i,R^\prime)$. Thus $\#B([x_i],R)\leq [\#B(x_i,R^\prime)]\leq\sup\{\#B(x,R^\prime)\mid x\in X\}$, hence the galaxies of $X^\omega$ are uniformly ULF. 

We denote the uniform Roe algebra (respectively, the uniformly quasi-local algebra) with respect to the coarse structure induced by its metric as $C^*_u(X,d)$ (respectively, $C^*_{ql}(X,d)$), and omit the metric $d$ unless necessary. An operator $T$ is in $\mathbb{C}_u[X,d]$ if and only if there exists $R>0$, such that for any $x,y\in X$ with $d(x,y)>R$, $(\delta_x,T\delta_y)=0$. We say $T$ has propagation no greater than $R$ in that case. The operators in $\mathbb{C}_u[X,d]$ are said to have \emph{finite propagation}.  

It is known that $C^*_{u}(X,\mathcal{E})\subset C^*_{ql}(X,\mathcal{E})$, and if $(X,\mathcal{E})$ is induced by a ULF metric structure with Yu's Property A, they are actually the same \cite{Roe=ql}.

Property A, introduced by Yu in \cite{Yu00}, became a core concept in coarse geometry soon after, and now has tons of equivalent descriptions. In the last section, we will use Property A implicitly by using Proposition 7.6 in \cite{SpW}, whose proof relies essentially on Property A.

\section{The universal uniformly locally finite coarse structure and the block diagonal algebra}


Let $X$ be a set. Recall that $\omega$ is a countably incomplete ultrafilter on an index set $I$. Then we have $\ell^2(X)\subset \ell^2(X^\omega)\subset (\ell^2(X))^{metric,\omega}$, where $ (\ell^2(X))^{metric,\omega}$ is the metric ultrapower of $ \ell^2(X)$. The second inclusion is given by $\delta_{[x_i]}\mapsto [\delta_{x_i}]$, where $\delta_x$ means the Dirac function at $x$. Since the set-theoretic ultrapower of  $\ell^2(X)$ is not involved in this paper, we (ab)use $(\ell^2(X))^{\omega}$ instead of $ (\ell^2(X))^{metric,\omega}$. For an infinite set $X$, all inclusions are strict. An operator $T\in B(\ell^2(X))$ extends to $T^\omega\in B( (\ell^2(X))^{\omega})$ by $T^\omega([v_i]):=[Tv_i]$, where the equivalence class is taken in metric ultrapower. 

Let $T\in B(\ell^2(X))$, and $k_T:X\times X \to \mathbb{C}$ the kernel function of $T$. We denote $k_T^\omega:X^\omega\times X^\omega\to \mathbb{C}^\omega$ the ultrapower of $k_T$. Let $P$ be the orthogonal projection of $(\ell^2(X))^{\omega}$ onto $ \ell^2(X^\omega)$. By calculation one can show that $(P[f_i])([x_i])=\lim_{i\to\omega} f_i(x_i)$. Viewing elements in $(l^2(X))^\omega$ as equivalence classes of (internal $\ell^2$) functions $X^\omega\to \mathbb{R}^\omega$, then $(Pf)(x)=st\circ f(x)$ for $f\in (l^2(X))^\omega$ and $x\in X^\omega$. The corner of $T^\omega$ on $\ell^2(X^\omega)$, that is $PT^\omega$ regarded as an operator in $B( \ell^2(X^\omega))$, has a kernel function $st \circ k_T^{\omega}$. This can be proved by checking
$$(\delta_{[x_i]},T^\omega \delta_{[y_i]})=st\circ [(\delta_{x_i},T \delta_{y_i})]=st\circ [k_T(x_i,y_i)]=st\circ k_T^\omega([x_i],[y_i]).$$
It is straightforward that $\ell^2(X)$ is an invariant subspace of $T^\omega$, while $\ell^2(X^\omega)$ may not be the case.

\begin{prop}

Let $T\in B(\ell^2(X))$ and $\omega$ a countably incomplete ultrafilter on $I$, the following are equivalent:

\begin{enumerate}

\item $\ell^2(X^\omega)$ is an invariant subspace of $T^\omega$;
\item $\forall \varepsilon >0$, $\exists M>0$ such that $\forall x\in X$, $\exists v\in \ell^2(X)$ with $\#supp(v)\leq M$ and $\|T\delta _x -v\|<\varepsilon$.

\end{enumerate}

\end{prop}
\begin{proof}

(2) implies (1): For any $[x_i]\in X^\omega$ and $\varepsilon >0$, there exists $M>0$ as in the assumption. Let $v_i\in \ell^2(X)$ such that $\#supp(v_i)\leq M$ and $\|T\delta _{x_i} -v_i\|<\varepsilon$. Then $\|T^\omega (\delta_{[x_i]})-[v_i]\|\leq \varepsilon$ and $[v_i]\in \ell^2(X^\omega)$ with finite support. Hence $T^\omega ([x_i])\in \ell^2(X^\omega)$, which implies that $ \ell^2(X^\omega)$ is an invariant subspace of $T$.

(1) implies (2): Suppose that (2) is not true. Then $\exists \varepsilon_0 >0$, $\forall n>0$, $\exists x_n\in X$, such that $\forall v\in \ell^2(X)$ with $\#supp(v)\leq n$, $\|T\delta _{x_n} -v\|\geq \varepsilon_0$. By countable incompleteness of $\omega$, decompose $I$ into disjoint union $\cup_{n\in \mathbb{N}} I_n$ with $I_n\notin \omega$ for all $n\in\mathbb{N}$. Define $y_i=x_n$ such that $i\in I_n$. Then for any $N\in \mathbb{N}$ and any $v\in \ell^2(X)$ with $\#supp(v)\leq N$, one has that
$T\delta_{y_i}$ is $\varepsilon_0$ away from $v$ for all $i\in I_n$ with $n\geq N$. Then for any $\xi\in \ell^2(X^\omega)$ with $supp(\xi)=K$ finite, let $\xi=\sum_{k=1}^{K}a^k\delta_{x^k}$, with $x^k=[x_i^k]\in X^\omega$. Then $\xi=[\sum_{k=1}^{K}a^k\delta_{x_i^k}]$. Hence $\|T^\omega\delta_{[y_i]}-\xi\|=st\circ [\|T\delta_{[y_i]}-\sum_{k=1}^{K}a^k\delta_{x_i^k}\|]\geq\varepsilon_0$. Thus, $T^\omega(\delta_{[y_i]})$ is $\varepsilon_0$ away from any finitely supported vector $\xi \in\ell^2(X^\omega)$, hence not in $ \ell^2(X^\omega)$.
\end{proof}

\begin{definition}

The block diagonal algebra is defined by
$$Q(X):=\left\{T\in B(\ell^2(X))\mid \ell^2(X^\omega)\text{ is an invariant subspace of }T^\omega\text{ and }{(T^*)}^\omega\right\}.$$

\end{definition}

Note that $T\in Q(X)$ if and only if $T^\omega$ is block diagonal with respect to $H=H_1\oplus H_2\oplus H_3$, where $H_1=\ell^2(X)$, $H_2$ is the orthogonal complement of $\ell^2(X)$ in $\ell^2(X^\omega)$, i.e., $\ell^2(X^\omega\setminus X)$, and $H_3$ is the orthogonal complement of $\ell^2(X^\omega)$ in $(\ell^2(X))^{\omega}$. Hence $Q(X)$ is a $C^*$-algebra and is independent of $\omega$ by Proposition 1. For $T\in Q(X)$, we denote the block of $T^\omega$ on $H_i$ as $T_i$, $i=1,2,3$. We call $T_2$ the \emph{universal limit operator} of $T$.

Now, we consider the relation between $Q(X)$ and the universal uniform quasi-local algebra $C^*_{ql}(X,\mathcal U)$. The following definition is due to \cite{STY}.

\begin{definition}
The \emph{universal uniformly locally finite (universal ULF) coarse structure} on $X$ is defined as
$$\mathcal{U}:=\{E\in X\times X\mid \exists n\in \mathbb{N}\text{ such that }\forall x \in X\text{, }\#E_x\leq n\text{ and }\#{}_xE\leq n)\}.$$
\end{definition}

The universal ULF coarse structure is the largest ULF coarse structure on $X$. We call the uniform Roe algebra (resp. uniform quasi-local algebra) with respect to universal ULF coarse structure on $X$ the uniform universal ULF Roe algebra (resp. quasi-local algebra) of $X$. It is the largest algebra among all uniform Roe algebras (resp. uniform quasi-local algebras) of $X$ with respect to ULF coarse structures.

\begin{prop}

If $T$ is in the uniform universal ULF quasi-local algebra on $X$, then $T\in Q(X)$.

\end{prop}

\begin{proof}

For any $\varepsilon>0$, since $T$ is in the universal ULF quasi-local algebra on $X$, there exists $M>0$ and $E\subset X\times X$ with every row and column no more than $M$ non-zero entries, such that for any $Y_1,Y_2\in X$ with $(Y_1\times Y_2) \cap E=\emptyset$, $\|P_{Y_1} T P_{Y_2}\|<\varepsilon$. Then for any $x\in X$, let $E_x:=\{y\in X\mid (y,x)\in E\}$. Then we have that $\|(P _{X\setminus E_x})T\delta_x\|<\varepsilon$ by quasi-locality. Hence, we can take $v=P_{E_x}T\delta_x$ so that it satisfies the condition given in Proposition 1 (2). The same argument applies to $T^*$.
\end{proof}

\begin{remark}
    Manuilov in \cite[Problem 4.3]{Ma} raised a question on whether several subalgebras of $B(\ell^2(X))$ are the same. Among them he asked whether $C^*_u(X,\mathcal{U})=Q(X)$. We have proved that $C^*_{ql}(X,\mathcal{U})\subset Q(X)$. If one can prove that $C^*_u(X,\mathcal{U})\subsetneq C^*_{ql}(X,\mathcal{U})$, then we have a negative answer to Manuilov's question. Ozawa \cite{Oz} proved that for ULF metric spaces containing an expander, the uniform Roe algebra is strictly contained in the uniform quasi-local algebra. However, as pointed out by the anonymous referee, the proof uses the countably generated property of the coarse structure in an essential and subtle way. Hence, Ozawa's proof cannot be generalized directly to non-metrizable cases, e.g., the universal ULF coarse structure case. This leads us to the following open question: are the inclusions $C^*_u(X,\mathcal{U})\subset C^*_{ql}(X,\mathcal{U})\subset Q(X)$ strict?
\end{remark}

Recall that a \emph{ghost operator} is an operator $T\in B(\ell^2(X))$ whose kernel function $k_T\in C_0(X\times X)$. By basic nonstandard analysis, $k_T\in C_0(X\times X)$ if and only if $st\circ k_T^\omega (x,y)=0$ for any $(x,y)\in (X^\omega\times X^\omega\setminus X\times X)$. Hence for $T\in Q(X)$, $T$ is a ghost operator if and only if $T_2=0$. And for Hilbert space $H$, an operator $T\in B(H)$ is compact if and only if $T^\omega v=0$ for all $v\in H^{\omega,metric}$ that is orthogonal to $H$, see \cite[Proposition 2.40]{Ng} for example. Hence $T\in B(\ell^2(X))$ is compact if and only if $T\in Q(X)$ and $T_2=0$, $T_3=0$.

\section{Limit operators on galaxies}

In this section, $(X,d)$ is a ULF metric space. Let $Q(X,d)$ be the set of operators in the block diagonal algebra $Q(X)$ such that $T_2$ is also block diagonal with respect to the decomposition
$$H_2=\ell^2(X^\omega \setminus X)=\bigoplus_{G\in AG(X^\omega)}\ell^2(G),$$
i.e., $T_2=\oplus_{G\in AG(X^\omega)}T_G$. Equivalently, $T\in Q(X,d)$ if and only if for any afar galaxy $G$, $\ell^2(G)$ is an invariant subspace of both $T^\omega$ and ${(T^*)}^\omega$. We call the block of $T_2$ on $\ell^2(G)$ the \emph{limit operator} of $T$ on the afar galaxy $G$, and denote it by $T_G$.

Similar to Proposition 1, we have the following proposition.

\begin{prop}

Let $(X,d)$ be a ULF metric space, $T\in B(\ell^2(X))$ and $\omega$ a countably incomplete ultrafilter on $I$, the following are equivalent:

\begin{enumerate}

\item For all afar galaxies $G$, $\ell^2(G)$ is an invariant subspace of $T^\omega$;
\item $\forall \varepsilon >0$, $\exists R>0$ such that $\forall x\in X$, $\exists v\in \ell^2(X)$ such that $d(supp(v),x)\leq R$ and $\|T\delta _x -v\|\leq\varepsilon$.

\end{enumerate}
\end{prop}

\begin{proof}

(2) implies (1): For any $[x_i]\in X^\omega$ and $\varepsilon >0$, there exists $R>0$ as in the assumption. Let $v_i\in \ell^2(X)$ such that $d(supp(v_i),x_i)\leq R$ and $\|T\delta _{x_i} -v_i\|<\varepsilon$. Then $\|T^\omega (\delta_{[x_i]})-[v_i]\|\leq \varepsilon$. Since $\sup _{i\in I}\#supp(v_i)<\infty$ by uniform local finiteness of $X$, we have $[v_i]\in \ell^2(X^\omega)$ and $supp([v_i])$ is contained in the ultraproduct of $supp(v_i)$. Thus $d^\omega(supp[v_i],[x_i])\leq R)$. This means that $T^\omega (\delta_{[x_i]})\in \ell^2(G_{[x_i]})$, where $G_{[x_i]}$ is the galaxy where $[x_i]$ dwells. This implies that $ \ell^2(G)$ is an invariant subspace of $T$ for all galaxies $G$.

(1) implies (2): Suppose that (2) is not true. Then $\exists \varepsilon_0 >0$, $\forall n>0$, $\exists x_n\in X$, such that $\forall v\in \ell^2(X)$ with $d(supp(v),x_n)\leq n$, one has $\|T\delta _{x_n} -v\|\geq \varepsilon_0$. By countable incompleteness of $\omega$, decompose $I$ into disjoint union $\cup_{n\in \mathbb{N}} I_n$ with $I_n\notin \omega$ for all $n\in\mathbb{N}$. Then for any $\xi\in \ell^2(X^\omega)$ with $\#supp(\xi)=K$ finite and $d^\omega(supp(\xi),[y_i])<n$, let $\xi=\sum_{k=1}^{K}a^k\delta_{x^k}$, with $x^k=[x_i^k]\in X^\omega$. Then $\xi=[\xi_i]=[\sum_{k=1}^{K}a^k\delta_{x_i^k}]$ with $\{i\in I\mid d(supp(\xi_i),y_i)\leq R\}\in \omega$. Hence $\|T^\omega\delta_{[y_i]}-\xi\|=st\circ [\|T\delta_{[y_i]}-\sum_{k=1}^{K}a^k\delta_{x_i^k}\|]\geq\varepsilon_0$. Thus, $T^\omega(\delta_{[y_i]})$ is $\varepsilon_0$ away from any finitely supported vector $\xi \in\ell^2(G_{[y_i]})$, hence not in $ \ell^2(G_{[y_i]})$.
\end{proof}

\begin{prop}
Let $T\in C^*_{ql}(X,d)$. Then for any finite subsets $A,B\subset X^\omega$ with $d^{\omega}(A,B)$ infinite, we have $\|P_{A} T^\omega P_{B}\|=0$.
\end{prop}

\begin{proof}
Since $T\in C^*_{ql}(X,d)$, for any $\varepsilon>0$, there exists $r(\varepsilon)>0$ such that for any finite subset $Y_1,Y_2 \subset X$, $d(Y_1,Y_2)>r(\varepsilon)$ implies that $\|P_{Y_1} T^\omega P_{Y_2}\|<\varepsilon$. Since $A$ is finite, it is an ultraproduct $A=[A_i]$ with $\#A_i=\#A$ for $\omega$-almost all $i$'s, and the same applies to $B=[B_i]$. Since $d^{\omega}(A,B)$ is infinite, $d(A_i,B_i)$ tends to infinity as $i\to\omega$. Hence $d(A_i,B_i)$ is eventually greater than any $r(\varepsilon)$, we have $\|P_{A_i} T P_{B_i}\|<\varepsilon$ eventually for any positive $\varepsilon>0$, hence tends to $0$, therefore $\|P_{A} T^\omega P_{B}\|=\lim _{i\to\omega}\|P_{A_i} T P_{B_i}\|=0$.
\end{proof}

Note that $C_{ql}^*(X,d) \subset C_{ql}^*(X,\mathcal{U}) \subset Q(X)$, Proposition 4 shows that $C_{ql}^*(X,d) \subset Q(X,d)$.  Since $T_2$ is block diagonal, the map from $T$ to $T_G$ is a $*$-homomorphism from $Q(X,d)$ to $B(\ell^2(G))$.

\begin{prop}

Let $G\in AG(X^\omega)$.
\begin{enumerate}

\item For $T\in \mathbb{C}_u[X,d]$, we have $T_G \in \mathbb{C}_u[G,st\circ d^{\omega}]$. 
\item For $T\in C^*_u(X,d)$, we have $T_G \in C^*_u(G,st\circ d^{\omega})$. 
\item For $T\in C^*_{ql}(X,d)$, we have $T_G \in C^*_{ql}(G,st\circ d^{\omega})$.
\item For $T\in Q(X,d)$, we have $T_G \in Q(G,st\circ d^{\omega})$. 

\end{enumerate}

\end{prop}

\begin{proof}
\begin{itemize}
\item[(1)] If $T$ has propagation no more than $R$, then for any $[x_i],[y_i]\in G$ with $st\circ d^\omega([x_i],[y_i])>R$, we have $(\delta_x,T_G\delta_y)=(\delta_x,T^\omega\delta_y)=\lim_{i\to \omega}(x_i,Ty_i)=0$. Hence $T_G$ also has propagation no more than $R$.
\item[(2)] By (1) and the fact that $*$-homomorphism is norm decreasing.
\item[(3)] Since $T\in C^*_{ql}(X,d)$, for any $\varepsilon>0$, there exists $r>0$ such that for any finite subsets $Y_1,Y_2 \subset X$, $d(Y_1,Y_2)>r$ implies that $\|P_{Y_1} T^\omega P_{Y_2}\|<\varepsilon$. Then for any finite subsets $A,B\subset G$ with $d^{\omega}(A,B)>r$, $\|P_{A} T_G P_{B}\|=\lim _{i\to\omega}\|P_{A_i} T P_{B_i}\|\leq \varepsilon$. Thus $T_G$ is quasi-local.
\item[(4)] For any $\varepsilon >0$, let $R$ be a coefficient of $T$ satisfying Proposition 3. Then for any $[x_i]\in G$, let $v_i\in \ell^2(X)$ with $d(supp(v_i),x_i)\leq R$ and $\|T\delta _{x_i} -v_i\|\leq\varepsilon$. Then $[v_i]\in \ell^2(G)$ with $st\circ d^\omega(supp([v_i]),[x_i])\leq R$ and $\|T_G\delta _{[x_i]} -[v_i]\|\leq\varepsilon$. The same proof applies to $T^*$, hence using Proposition 3 we have $T_G\in Q(G,st\circ d^{\omega})$.\qedhere
\end{itemize}\end{proof}

Now we are ready to prove Theorem 1 (1).

\begin{theorem}

Let $Ghost(X)$ be the set of ghost operators on $B(\ell^2(X))$ and $A$ a $C^*$-subalgebra of $Q(X)$. 

\begin{enumerate}
\item $Ghost(X)\cap A$ is a two-sided closed ideal in $A$.
\item An operator $T\in A$ is a generalized Fredholm operator with respect to the ghost ideal in $A$ if and only if its universal limit operator $T_2$ is invertible.
\item If $A$ is also a $C^*$-subalgebra of $Q(X,d)$ where $d$ is a ULF metric on $X$, then $T\in A$ is a generalized Fredholm operator with respect to the ghost ideal in $A$ if and only if there exists $M>0$ such that for any afar galaxy $G$, $T_G$ is invertible and $\|T_G^{-1}\|<M$.
\end{enumerate}

\end{theorem}

\begin{proof}
Since the kernel of the $*$-homomorphism $T \mapsto T_2$ is $Ghost(X)\cap A$, we have (1) and (2). For (3), if $A$ is also a $C^*$-subalgebra of $Q(X,d)$, $T_2$ is block diagonal, i.e., $T_2=\oplus _{G\in AG(X^\omega)}T_G$. Hence $T_2$ is invertible if and only if there exists $M>0$ such that for all afar galaxy $G$, $T_G$ is invertible and $\|T_G^{-1}\|<M$.
\end{proof}


\section{Hyper-principal map on ultrapowers of ULF metric spaces}

This section is devoted to the relationship between our setting and that in \cite{SpW}.
Let $(X,d)$ be a ULF metric space, and $\beta X$  the Stone-\v{C}ech compactification of $X$. For $x,y \in \beta X$, define $x\sim y$ if there is a function $f:X\to X$ such that $\sup\{d(p,f(p))\mid p\in X\}<\infty$ and its continuous extension $\bar{f}:\beta X \to \beta X$ maps $x$ to $y$. 

\begin{prop}

Let $(X,d)$ be a ULF metric space and $f:X\to X$ a function such that $\sup\{d(p,f(p))\mid p\in X\}<\infty$. Let $U$ be an ultrafilter on $X$, then there exists $Y\in U$ such that $f|_Y$ is injective.

\end{prop}

\begin{proof}
Let $\sup\{d(p,f(p))\mid p\in X\}=R$, and $N=\sup\{\#\overline{B(x,R)}\mid x\in X\}$. For every $x\in X$, enumerate $\overline{B(x,R)}$ as $x_1,x_2,\cdots,x_{N_x}$. Then we have $N_x\leq N$. For any $y\in X$, let $x=f(y)$, then $y=x_n$ for some $1\leq n \leq N_x$. Let $Y_n=\{y\in X\mid y=x_n\text{ for }x=f(y)\}$, then $\cup _{1\leq n\leq N}Y_n=X$. Thus there exists $n$ such that $Y_n\in U$. Since $f|_{Y_n}^{-1}(x)=x_n$ for $x$ in the image of $f|_{Y_n}$, we conclude that $f|_{Y_n}$ is injective .
\end{proof}

Proposition 6 shows that we can additionally demand that $f$ is injective on some $Y\in U$, which is in the original setting of \cite{SpW}. Since $f(Y)\in \bar{f}(x)$ (see Proposition 7 below) , the function $f^{-1}:f(Y)\to Y$ shows that $y\sim x$, i.e., the relation $\sim$ is an equivalence relation on $\beta X$. A limit space is an equivalence class under $\sim$, with pseudo-metric $d(x,y)=\lim_{p\to x} d(p,f(p))$, independent of the choice of $f$ whose extension maps $x$ to $y$. Note that here $X$ does not need to be strongly discrete, at the cost that a limit space will be a pseudo-metric space, which slightly generalizes the original definition in \cite{SpW} (see Remark 2). For $U \in \beta X$, denote $X_U$ the limit space where $U$ dwells.  

For a discrete topological space $X$, $\beta X$ can be realized as the set of all ultrafilters on $X$, with the topology generated by clopen sets $O_A:=\{U\in\beta X\mid A\in U\}$ with $A \subset X$. Define the hyper-principal map $\pi:X^\omega \to \beta X$ by $\pi (x):=\{A \subset X\mid x\in A^\omega\}$. The following proposition gives alternative descriptions of the hyper-principal map.

\begin{prop}

For a function $x:I\to X$, the following elements in $\beta X$ coincide:

\begin{enumerate}
\item $\pi([x_i])$, the image of $[x_i]\in X^\omega$ under the hyper-principal map;
\item $x_*(\omega)=\{A\subset X\mid x^{-1}(A)\in \omega\}$, the push forward of $\omega$ by the function $x$;
\item $\bar{x}(\omega)$, where $\bar{x}$ is the continuous extension of $x:I\to X$ to $\beta I \to \beta X$.\qed

\end{enumerate}

\end{prop}

We need the following proposition (see, for example, \cite[Theorem 1.3.10]{Go22}) for the next theorem.

\begin{prop}
Let $U$ be an ultrafilter on the set $S$. If $\bar{f}(U)=U$ for a function $f:S\to S$, then there exists $Y\in U$ such that $f|_Y$ is the identity map on $Y$.
\end{prop}
 
Note that the proof of the above proposition uses the same lemma as in \cite[Lemma 3.5]{SpW}, and can be used to prove \cite[Lemma 3.4]{SpW}, which is the key lemma to define the metric on limit spaces. 

Now we are ready to prove the main theorem of this section.

\begin{theorem}

Let $G\in AG(X^\omega)$. Then $\pi(G)$ is a limit space in $\beta X$, and $\pi |_G: G\to \pi(G)$ is a bijective isometry between pseudo-metric spaces.

\end{theorem}

\begin{proof}

Let $[x_i],[y_i]\in X^\omega$ be in the same galaxy, i.e., for some positive real number $R$, $d^\omega([x_i],[y_i])<R$ . 

We first prove that $\pi([x_i])\sim \pi([y_i])$. Let $N\in \mathbb{N}$ such that $\#B(x,R)<N$ for all $x\in X$. For each $x\in X$, fix an enumeration of $B(x,R)$ as $x_1,x_2,\cdots,x_{N_x}$. Since $d^\omega([x_i],[y_i])<R$, for $\omega$-almost all $i$'s, $y_i=x_{i,n_i}$ for some $1\leq n_i \leq N_x\leq N$. Let $n=[n_i]$, then $n\in \mathbb{N}$, and $y_i=x_{i,n}$ for $\omega$-almost all $i$'s. Define $f:X\to X$ by mapping $x$ to $x_n$, then there is $J\in\omega$ and $f(x_i)=y_i$ for $i\in J$. Thus $\bar{f}(\pi([x_i]))=\bar{f}\bar{x}(\omega)=\overline{f\circ x}(\omega)=\bar{y}(\omega)=\pi([y_i])$, hence $\pi([x_i])\sim\pi([y_i])$. This shows that $\pi(G)$ is contained in a limit space in the sense of \cite{SpW}. 

To prove that $\pi|_G$ is injective, let $\pi([x_i])=\pi([y_i])$. Since $\bar{f}(\pi([x_i]))=\pi([y_i])$, there exists $K\in\pi([x_i])$ such that $f|_K=id_K$. Hence for $i\in J\cap x^{-1}(K)$, $x_i=y_i$, hence $[x_i]=[y_i]$. 

To see that $\pi |_G$ is an isometry, we directly check
$$\lim\limits_{x\to \pi([x_i])}d(x,f(x))=\lim\limits_{i\to \omega}d(x_i,y_i)=st\circ d^\omega([x_i],[y_i]),$$
where $f$ is as above, satisfying $f(x_i)=y_i$ for $i\in J$ for some $J\in\omega$.

Now we prove that $\pi(G)$ is a whole limit space. Let $[x_i]\in X^\omega$. For any $U\in X_{\pi([x_i])}$, there exists $f:X \to X$ such that $\sup\{d(p,f(p))\mid p\in X\}<\infty$ and $\bar{f}(\pi([x_i]))=U$. Then
$$U=\bar{f}(\pi([x_i]))=\bar{f}\circ\bar{x}(\omega)=\overline{f\circ x}(\omega)=\pi([f(x_i)]).$$
Hence $U$ has a pre-image in $G$.
\end{proof}

\begin{remark}
    In the original setting of \cite{SpW}, $(X,d)$ needs to be strongly discrete, i.e. for any $x,y\in X$, $d(x,y)\in \mathbb{Z}$. This implies that $\inf\{d(x,y)\mid x,y\in X,x\neq y\}>0$, hence $d^\omega(x,y)$ is an infinitesimal only if it is actually zero. In this case, $(G,st \circ d^\omega)$ is a metric space, thus $\pi$ gives rise to an isometry of metric spaces between a galaxy and its corresponding limit space.
\end{remark}


It is straightforward to check that if $I=X$, $\omega\in\beta X$, then $[id]\in X^\omega$, where $id$ is the identity map on $X$, and $\pi$ maps the galaxy where $[id]$ dwells to $X_\omega$, the limit space associated to $\omega$.

Fix $T\in C^*_u(X)$. For $x,y\in X^\omega$ with $d^\omega(x,y)$ finite, let $x=[x_i]$ and $y=[y_i]$. Then we have $st\circ k_T^{\omega}(x,y)=\lim_{i\to \omega}k_T(x_i,y_i)=\bar{k}_{T}(\pi(x,y))$, here we abuse a little bit the notation of $\pi$ as the hyper-principal map from $(X\times X)^\omega \to \beta(X\times X)$. Note that $\bar{p}_1((\pi(x,y))=\pi(x)$ and $\bar{p}_2((\pi(x,y))=\pi(y)$, where $p_i$ is the projection of $X\times X$ onto its $i$'th coordinate. Since $X$ is ULF, the canonical quotient map $\beta(X\times X)\to \beta X\times \beta X$ restricts to a homeomorphism on $\overline E$, where $E$ is an entourage \cite{Roe2003}. Thus, $\bar{k}_{T}(\pi(x,y))$ only depends on $\pi(x)$ and $\pi(y)$, which is the value of kernel function of the limit operator at $(\pi(x),\pi(y))$ in \cite{SpW}. This shows that the two definitions of limit operators also coincide.

\begin{remark}

In contrast with the good behavior of $\pi|_G$, the hyper-principal map $\pi$ is not necessarily injective, nor necessarily surjective (see, for example, \cite[Chapter 9]{Go22}). Thus, a limit space might have no copies, or many copies in $X^\omega$. Recall that the \emph{Rudin-Keisler order} is a partial pre-order on the class of ultrafilters, such that for ultrafilters $U$ on $I$ and $V$ on $J$, $U\leq _{RK} V$ if and only if there exists a function $f: J\to I$ such that $\bar{f}(V)=U$. Then by definition, $U$ is in the image of $\pi$ if and only if $U\leq _{RK} \omega$. Hence, we can choose $\omega$ to be Rudin-Keisler greater than any ultrafilter on $X$ (though $I$ cannot have the same cardinality as $X$ in this case) to make $\pi$ surjective, i.e., every limit space is contained in $X^\omega$. On the other hand, if we choose a countably incomplete ultrafilter that is not the case (for example, $I=X$), then there will be some limit spaces without any copies as galaxies, which means that our theorem improves the result in \cite{SpW} by requiring invertibility of fewer limit operators.

\end{remark}

\section{Property A and uniform boundedness of $T_G^{-1}$}

In this section, we prove Theorem 1 (2). Let $(X,d)$ be a ULF metric space with Yu's Property A. Fix a $p \in X$ and an operator $T\in C^*_u(X)$. Following \cite{SpW}, let $\nu (T)$ be the lower norm of $T$, i.e., $\nu (T):=\inf \{\|Tv\| \mid \|v\|=1\}$. Fix $w\in \overline{\{\nu (T_G)\mid G\in AG(X^\omega) \}}$. We wish to prove that $w\in\{\nu (T_G)\mid G\in AG(X^\omega) \}$, that is, $\{\nu (T_G)\mid G\in AG(X^\omega) \}$ is closed. This is a consequence of the following result.

\begin{prop}

There is a sequence of real numbers $(r_n)_{n\in \mathbb{N}}$ tending to infinity such that there exists $x\in X^{\omega}$ such that 

\begin{enumerate}

\item $d^{\omega}(x,p)$ is infinite, and
\item For all $n\in \mathbb{N}$, $|\nu (T^{\omega}P_{B(x,r_n)})-w|\leq 1/2^n$.

\end{enumerate}

\end{prop}

The proposition tells that there is an afar element $x$, such that the lower norm of $T^\omega$ truncated by a finite radius ball centered at $x$ tends to $w$ as the radius tends to infinity. It is straightforward that, if the above proposition is true, i.e., there is an $x$ satisfying the proposition, then $v(T_G)=w$, where $G$ is the galaxy in which $x$ dwells. 

Note that the proposition gives countably many conditions about an element in $X^\omega$, and states that there is an $x$ satisfying all the conditions. Inspired by countable saturation of ultrapower, to prove the proposition, one only needs to show that any finite set of the conditions can be satisfied simultaneously.

\begin{lemma}

There is an infinite hyperreal $R$ and a sequence of real numbers $(r_n)_{n\in \mathbb{N}}$ tending to infinity, such that for any $n\in \mathbb{N}$ there exists $x_n\in X^{\omega}$ such that 

\begin{enumerate}

\item $d^{\omega}(x_n,p)\geq R$, and
\item For all $k \leq n$, $|\nu (T^{\omega} P_{B(x_n,r_k)})-w|<1/2^k$.

\end{enumerate}
\end{lemma}

Note that $\nu (T^{\omega} P_{B(x_n,r_k)})$ can be formulated as the standard part of a hyperreal defined in the nonstandard universe, so that we can apply the saturation argument to prove Proposition 9 from Lemma 1. For those who are not familiar with the saturation property, we give a direct proof here, which is an adaptation of the proof of countable saturation of ultraproduct, via diagonal argument. 

We first note that finitely truncated lower norm commutes with ultraproducts.

\begin{prop}

A finite set $F\subset X^\omega$ is always an ultraproduct $F=[F_i]$ with $\#F_i=\#F$ for $\omega$-almost all $i$'s, and $\nu (T^\omega P_{F})=\lim _{i\to \omega}\nu(TP_{F_i})$.\qed

\end{prop}

Now we prove Proposition 9 using Lemma 1. In the following proof, the index $i\in I$ will be written as functions (e.g., $R=[R(i)]$) rather than a subindex.

\begin{proof}[Proof of Proposition 9]

For each $n\in\mathbb{N}$, let $x_n=[x_n(i)]$ satisfy the lemma. Let $I_n$ be the set of $i$'s satisfying the following properties:

\begin{enumerate}

\item $d(x(i),p)\geq R(i)$,
\item For every $k \leq n$, $|\nu (T P_{B(x_n(i),r_k)})-w|<1/2^k$.
\end{enumerate}

By Proposition 10, $I_n\in \omega$. By countable incompleteness of $\omega$, there exists a decreasing sequence $(J_n)_{n\in\mathbb{N}}$ such that every $J_n \in \omega$ and $\cap J_n=\phi$. Let $K_n=(\cap _{k=0}^{n}I_k) \cap J_n$, then every $K_n \in \omega$ and $K_n$ is decreasing with $\cap K_n=\phi$. 
Define $x=[x(i)]\in X^\omega$ by the following steps:

Step 0: Let $x=x_0$.

Step $n$ for $n\geq 1$: For $i\in K_n$, re-value $x(i)$ to be $x_n(i)$.

For every $i$, $x(i)$ will not be re-valued after finite steps, hence $x=[x(i)]$ is well-defined. By its construction, it satisfies the following properties:

\begin{enumerate}

\item $d(x(i),p)\geq R(i)$ for $i\in K_0$, and
\item For any $n\in\mathbb{N}$, $|\nu (T P_{B(x(i),r_n)})-w|<1/2^n$ for $i\in K_n$.
\end{enumerate}

Thus $x=[x(i)]$ satisfies the proposition.
\end{proof}

Before proving Lemma 1, we need \cite[Proposition 7.6]{SpW}:

\begin{prop}[\cite{SpW}]
Let $(X,d)$ be a ULF metric space with Property A and $T\in C^*_u(X)$. Then for any $\delta >0$, there exists $s>0$ such that for any finite $F \subset X$, there exists a non-empty $Y \subset F$ with $diam(Y)\leq s$ such that $0 \leq  \nu (T P_Y) - \nu (T P _F)\leq \delta$.
\end{prop}

The proposition says that truncations of $T$ can be truncated further by uniformly bounded sets, with their lower norms changing arbitrarily small. Next proposition shows that the same holds for limit operators of $T$ with the same coefficients:

\begin{prop}
For any $\delta >0$, let $s$ be as in Proposition 11. Then for any galaxy $G$ of $X^\omega$, and finite set $F\subset G$, there exists a non-empty $Y \subset F$ with $diam(Y)\leq s$ such that $0\leq \nu(T_G P_Y)-\nu(T_G P_F)\leq \delta$.
\end{prop}

\begin{proof}
Let $F=[F_i]$, for each $F_i$, by using the above proposition, we get $Y_i$. Let $Y=[Y_i]$. By Proposition 11, we have $0\leq \nu(T^\omega P_Y)-\nu(T^\omega P_F)\leq \delta$. Notice that $T_G P_F=T^\omega P_F$ and $T_G P_Y=T^\omega P_Y$, we finish the proof.
\end{proof}

To apply Proposition 12 to $T_G$, we should first truncate $T_G$ on some finite set:

\begin{prop}
For any $\delta >0$, let $s$ be as in Proposition 11. Then for any galaxy $G$ of $X^\omega$ and $\delta^\prime >\delta$, there exists non-empty set $Y\subset G$ with $diam(Y)\leq s$ such that $0\leq \nu(T_G P_Y)-\nu(T_G)<\delta^\prime$.
\end{prop}

\begin{proof}

Find $F\subset G$ finite such that $0\leq \nu(T_G P_F)-\nu(T_G)<\delta^\prime-\delta$. Then use the above proposition to find a $Y\subset F$ such that $0\leq \nu(T_G P_Y)-\nu(T_G P_F)\leq \delta$. Then $0\leq \nu(T_G P_Y)-\nu(T_G)<\delta^\prime$.
\end{proof}

It is time to prove Lemma 1.

\begin{proof}[Proof of Lemma 1]

Let $G_n$ be an afar galaxy with $|\nu(T_{G_n})-w|<1/2^{n+1}$. As in Section 2.1, there is an infinite positive hyperreal $R$ such that for any $G_n$ and any element $x\in G_n$, $d^\omega(p,x)\geq R$. By the Proposition 13, let $s_n$ be a coefficient $s$ for $\delta=1/(1.001\times 2^{n+1})$. Take a sequence $r_n\geq s_n$, monotonically increasing to infinity. For any $n\in \mathbb{N}$, construct $Y_n$ from $G_n$ using Proposition 13, then build $Y_k$ from $Y_{k+1}$ using Proposition 12 inductively. Then we get a sequence $Y_0\subset Y_1\subset \cdots \subset Y_n \subset G_n$ such that 

\begin{enumerate}

\item $diam(Y_k)\leq r_k$,
\item $0\leq\nu (T_{G_n} P_{Y_n})-\nu (T_{G_n}) < 1/2^{n+1}$,
\item For $0\leq k \leq n-1$, $0\leq \nu (T_{G_n} P_{Y_k})-\nu (T_{G_n} P_{Y_{k+1}})<1/2^{k+1}$.

\end{enumerate}
In that case, $\nu (T_{G_n} P_{Y_k})-\nu (T_{G_n})<1/2^{k+1}+1/2^{k+2}+\cdots +1/2^{n+1}=1/2^k-1/2^{n+1}$.
Choose $x_n \in Y_0$, then $0\leq \nu (T_{G_n} P_{B(x_n,r_k)})-\nu (T_{G_n})\leq \nu (T_{G_n} P_{Y_k})-\nu (T_{G_n})<1/2^k-1/2^{n+1}$. Hence $|\nu (T_{G_n} P_{B(x_n,r_k)})-w|<1/2^k$.

Hence let $R$ and $r_n$ be as above, then for any $n\in \mathbb{N}$, we find an $x_n$ satisfying the lemma. 
\end{proof}

Thus, we prove Lemma 1, hence Proposition 9. By Proposition 9,  $\{\nu (T_G)\mid G\in AG(X^\omega)\}$ is closed. Hence, if $T_G$ is invertible for all afar galaxies $G$, then $T_G^{-1}$ is automatically uniformly bounded. If not, since if $\|T_G^{-1}\|$ are unbounded, $0$ will be in the closure of $\{\nu (T_G)\mid G\in AG(X^\omega)\}$, hence in $\{\nu (T_G)\mid G\in AG(X^\omega)\}$ itself. Therefore, there exists $G$ such that $\nu(T_G)=0$, contradicting the statement that all $T_G$ are invertible. Thus, the uniform boundedness condition in Theorem 1 (1) can be removed in the case of Yu's Property A. Note that the ghost ideal in $C_u^*(X,d)$ is equal to the set of compact operators by Property A of $(X,d)$, therefore Theorem 1 (2) follows. 

\section*{Acknowledgement}

The authors would like to thank Isaac Goldbring for inspiring discussions and advice. We also thank Jintao Deng for his careful reading of the draft of this paper. We are also grateful to the anonymous referee for many valuable suggestions, especially in Remark 1.

This work was partially supported by NSFC (No. 12571135, 12171156), Key Laboratory of MEA (Ministry of Education), the Science and Technology Commission of Shanghai (No. 22DZ2229014), and Shanghai Key Laboratory of PMMP, East China Normal University.

\bibliography{ref.bib}
\bibliographystyle{alpha}

\end{document}